\documentclass[12pt,a4paper]{article}
\usepackage[small,sc]{titlesec}
\usepackage{amssymb}
\usepackage{amsthm}
\usepackage{graphicx}
\usepackage{amsmath}
\theoremstyle{plain}
\newtheorem{thm}{\sc \bf{Theorem}}[section]
\newtheorem{lem}[thm]{\sc \bf{Lemma}}

\title{\bf{On Greedy Clique Decompositions and Set Representations of Graphs}\footnote{
Supported partially by the National Science Council under Grants NSC
96-2115-M-029-001 and NSC 96-2115-M-029-007}}
\author{Tao-Ming Wang\footnote{E-mail address: wang@thu.edu.tw}~ and Jun-Lin Kuo \\\\
Department of Mathematics \\
Tunghai University\\
Taichung, Taiwan 40704 }
\date{\today}
\begin{document}
\maketitle

\begin{abstract}
%In 1966 P. Erd\"{o}s, A. Goodman, and L. P\'{o}sa proved (The
%representation of a graph by set intersections, \textit{Canadian J.
%Math.} \textbf{18} (1966) 106-112) that the edge set of an
%$n$-vertex simple graph $G$ without isolated vertices can be
%partitioned using at most $\lfloor n^2/4 \rfloor$ cliques.
%A couple
%of tens of years behind McGuinness \cite{mcguinness} proved that any
%greedy clique partition is such a partition.
In 1994 S. McGuinness showed that any greedy clique decomposition of
an $n$-vertex graph has at most $\lfloor n^2/4 \rfloor$ cliques (The
greedy clique decomposition of a graph, \textit{J. Graph Theory}
\textbf{18} (1994) 427-430), where a \textit{clique decomposition}
means a clique partition of the edge set and a \textit{greedy clique
decomposition} of a graph is obtained by removing maximal cliques
from a graph one by one until the graph is empty. This result solved
a conjecture by P. Winkler. A \textit{multifamily set
representation} of a simple graph $G$ is a family of sets, not
necessarily distinct, each member of which represents a vertex in
$G$, and the intersection of two sets is non-empty if and only if
two corresponding vertices in $G$ are adjacent. It is well known
that for a graph $G$, there is a one-to-one correspondence between
multifamily set representations and clique coverings of the edge
set.
%In fact, if we define multifamily representation of a
%loopless multigraph $M$ to be a family of sets for which each member
%represents a vertex in $M$, and the two vertices are adjacent with
%$q$ edges in $M$ if and only if the corresponding representation
%sets have an intersection of cardinality $q$.
%and the number of elements in the intersection
%of two members of which represent the number of edges between the
%two corresponding vertices in $M$,
Further for a graph one may have a one-to-one correspondence between
particular multifamily set representations with intersection size at
most one and clique partitions of the edge set. In this paper, we
study for an $n$-vertex graph the variant of the set representations
using a family of distinct sets, including the greedy way to get the
corresponding clique partition of the edge set of the graph.
Similarly, in this case, we obtain a result that any greedy clique
decomposition of an $n$-vertex graph has at most $\lfloor n^2/4
\rfloor$ cliques.
%In particular we obtain the optimal upper bound for the
%corresponding family intersection numbers of graphs.

%In section 1 we will narrate this correspondence in full detail. If
%a multifamily representation of a multigraph $M$ has its members
%pairwise distinct, then it is designated as a representation of $M$.

%In section 2 we turn the aforementioned correspondence between
%multifamily representations and clique partitions of $M$ to account
%to prove that any multigraph $M$ with at most one edge between any
%two vertices of it can be represented by at most $\lfloor n^2/4
%\rfloor$ elements and we can accomplish such a representation from
%any greedy clique partition by a straightforward method based on
%this correspondence.
\end{abstract}

\section{\bf{Background and Introduction}}

By an \textit{multigraph} $M=(V(M),E(M),q)$ we mean a triple
consisting of a set $V(M)$ of \textit{vertices}, a set $E(M)$ of
\textit{edges}, and an integer-valued function $q$ defined on
$V(M)\times V(M)$ in the following way. For each unordered pair
$\{u,v\} \subset V(M)$, let $q(u,v)$ be the number of
\textit{parallel edges} joining $u$ with $v$. If $q(u,v)\neq 0$,
then we say that $\{u,v\}$ is an \textit{edge} of $M$ and $q(u,v)$
is called the \textit{multiplicity} of the edge $\{u,v\}$. For the
main results in this paper, we consider only \textit{finite,
undirected, simple} multigraphs, where \textit{simple} means that
$q(u,v)\leq 1$ for every $\{u,v\}\subset V$ and $q(u,u)=0$ for every
$u\in V(M)$. Therefore we simply call such multigraphs to be
\textit{graphs} for short throughout this article, unless otherwise
stated.

For a vertex subset $S\subseteq V(M)$, $\langle S\rangle_V$ denotes
the \textit{subgraph induced by} $S$. For a vertex $v$ in $M$,
$d_M(v)$ or $d(v)$ denote the \textit{degree} of $v$ in $M$. Let
$\mathcal{F}=\{S_1,...,S_p\}$ be a \textit{family} of distinct
nonempty subsets of a set $X$. Then $\textbf{S}(\mathcal{F})$
denotes the union of sets in $\mathcal{F}$. The \textit{intersection
multigraph} of $\mathcal{F}$, denoted $\Omega(\mathcal{F})$, is
defined by $V(\Omega(\mathcal{F}))=\mathcal{F}$, with $|S_i\cap
S_j|=q(S_i,S_j)$ whenever $i\neq j$. Of course, so long as we are
involved in this paper, $|S_i\cap S_j|$ always equal either $0$ or
$1$ for all $i\neq j$, as appointed above.

We say that a multigraph $M$ is an intersection multigraph on a
family (a multifamily, respectively) $\mathcal{F}$, if there exist a
family (a multifamily, respectively)  $\mathcal{F}$ such that
$M\cong \Omega(\mathcal{F})$. We say that $\mathcal{F}$ is a
\textit{representation} (a \textit{multifamily representation}
respectively) of the multigraph $M$. The \textit{intersection
number}, denoted $\omega (M)$ (\textit{multifamily intersection
number}, denoted $\omega_{m}(M)$, respectively), of a given
multigraph $M$ is the minimum cardinality of a set $X$ such that $M$
is an intersection multigraph (\textit{multifamily intersection
multigraph}, respectively) on a family (a multifamily, respectively)
$\mathcal{F}$ consisting of distinct (not necessarily distinct,
respectively) subsets of $X$. In this case we also say that
$\mathcal{F}$ is a \textit{minimum representation}
(\textit{multifamily representation}, respectively) of $M$.

Note that given a representation $\{S_v\mid v\in V(M)\}$ of $M$ and
a vertex subset $S\subseteq V(M)$, then $\{S_v\mid v\in S\}$ form a
representation of $\langle S \rangle_V$. Thus we know that
$\omega(M)$ is not less than $\omega(\langle S \rangle_V)$ for any
$S\subseteq V(M)$. Similarly for $\omega_{m}(M)$.

In 1966 P. Erd\"{o}s et al. \cite{erdos} proved that the edge set of
any simple graph $G$ with $n$ vertices, no one of which is isolated
vertex, can be partitioned using at most $\lfloor n^2/4 \rfloor$
cliques. In a couple decades S. McGuinness \cite{mcguinness} showed
that any greedy clique partition is such a partition.

A multifamily representation of a graph $G$ is a family of sets each
member of which represent a vertex in $G$ and the intersection
relation of two members of which represent the adjacency of the two
corresponding vertices in $G$. P. Erd\"{o}s et al. \cite{erdos}
suggested a one-one correspondence between multifamily
representations and clique coverings of a graph $G$. In fact, we may
define a \textit{multifamily representation of a multigraph} $M$ to
be a family of sets for which each member represents a vertex in
$M$, and the two vertices are adjacent with $q$ edges in $M$ if and
only if the corresponding representation sets have an intersection
of cardinality $q$. Then there is also a one-one correspondence
between multifamily representations and clique partitions of $M$.

In next section we will narrate this correspondence in full detail.
If a multifamily representation of a multigraph $M$ has pairwise
distinct member sets, then it is called a \textit{representation} of
$M$. And then we turn the correspondence to prove that any
$n$-vertex graph can be represented by at most $\lfloor n^2/4
\rfloor$ elements and we can accomplish such a representation from
any greedy clique partition by a straightforward method based on
this correspondence. In the end, certain future directions will be
mentioned.

\section{\bf{Partition Edge Set by Cliques}}

Given a multigraph $M=(V(M), E(M), q)$, $Q\subseteq V(M)$ is said to
be a \textit{clique} of $M$ if every pair of distinct vertices $u,v$
in $Q$ has $q(u,v)\neq 0$. A \textit{clique partition} $\mathcal{Q}$
of a multigraph is a set of cliques such that every pair of distinct
vertices $u,v$ in $V(M)$ simultaneously appear in exactly $q(u,v)$
cliques in $\mathcal{Q}$ and for each \textit{isolated vertex}, that
is, vertex with no edge incident to it, we need to use at least one
\textit{trivial clique}, that is, clique with only one vertex, in
$\mathcal{Q}$ to cover it. The minimum cardinality of a clique
partition of $M$ is called the \textit{clique partition number} of
$M$, and is denoted by $cp(M)$. This number must exist as the edge
set of $M$ forms a clique partition for $M$. We refer to a clique
partition of $M$ with the cardinality $cp(M)$ as a \textit{minimum
clique partition} of $M$.

Note that a clique partition $\mathcal{Q}$ of $M$ give rise to a
clique partition of $M-v$ by deleting the vertex $v$ from each
clique in $\mathcal{Q}$. Thus $cp(M)$ is not less than the clique
partition number of any induced subgraph of $M$.

%Given a multigraph $M=(V(M);q)$, $Q\subseteq V(M)$ is said to be a
%\textit{clique} of $M$ if every pair of distinct vertices $u,v$ in
%$Q$ has $q(u,v)\neq 0$. A \textit{clique partition} $\mathcal{Q}$ of
%a multigraph is a set of cliques such that every pair of distinct
%vertices $u,v$ in $V(M)$ simultaneously appear in exactly $q(u,v)$
%cliques in $\mathcal{Q}$ and for each \textit{isolated vertex}, that
%is, vertex with no edge incident to it, we need to use at least one
%\textit{trivial clique}, that is, clique with only one vertex, in
%$\mathcal{Q}$ to cover it. The minimum cardinality of a clique
%partition of $M$ is called the \textit{clique partition number} of
%$M$, and is denoted by $cp(M)$. This number must exist as the edge
%set of $M$ forms a clique partition for $M$. We refer to a clique
%partition of $M$ with the cardinality $cp(M)$ as a \textit{minimum
%clique partition} of $M$.
%
%Note that a clique partition $\mathcal{Q}$ of $M$ give rise to a
%clique partition of $M-v$ by deleting the vertex $v$ from each
%clique in $\mathcal{Q}$. Thus $cp(M)$ is not less than the clique
%partition number of any induced subgraph of $M$.

P. Erd\"{o}s et al. \cite{erdos} proved the following theorem.

\begin{thm}
The edge set of any simple graph $G$ with $n$ vertices no one of
which is isolated vertex can be partitioned using at most $\lfloor
n^2/4 \rfloor$ triangles and edges, and that the complete bipartite
graph $K_{\lfloor n/2 \rfloor, \lceil n/2 \rceil}$ gives equality.
\end{thm}

We somewhat modify their proof to prove the following theorem. We
use $G^{(n)}$ to denote a graph $G$ with $n$ vertices.

\begin{thm}\label{similarErdos}
Any graph $G$ with $n\geq 4$ vertices (perhaps with isolated
vertices) can be partitioned with at most $\lfloor n^2/4 \rfloor$
cliques $Q_1,...,Q_N$ such that for any two vertices $u,v$ in $G$,
we have

\begin{align*}
\{Q_i\mid u \in Q_i &\in \{Q_1,...,Q_N\}\}\\
\quad &\neq \quad \{Q_i\mid v\in Q_i\in \{Q_1,...,Q_N\}\}\tag{1}.
\end{align*}
Note that in such partition we need only use edges and triangles.
Furthermore, the upper bound $\lfloor n^2/4 \rfloor$ is optimal.
\end{thm}

\begin{proof}
For $n=4$, it is easy to check the theorem holds for the 11
different graphs on 4 vertices. (Please see Figure~\ref{v4})

\begin{figure}[h]
\centering
\includegraphics[width=0.7\textwidth]{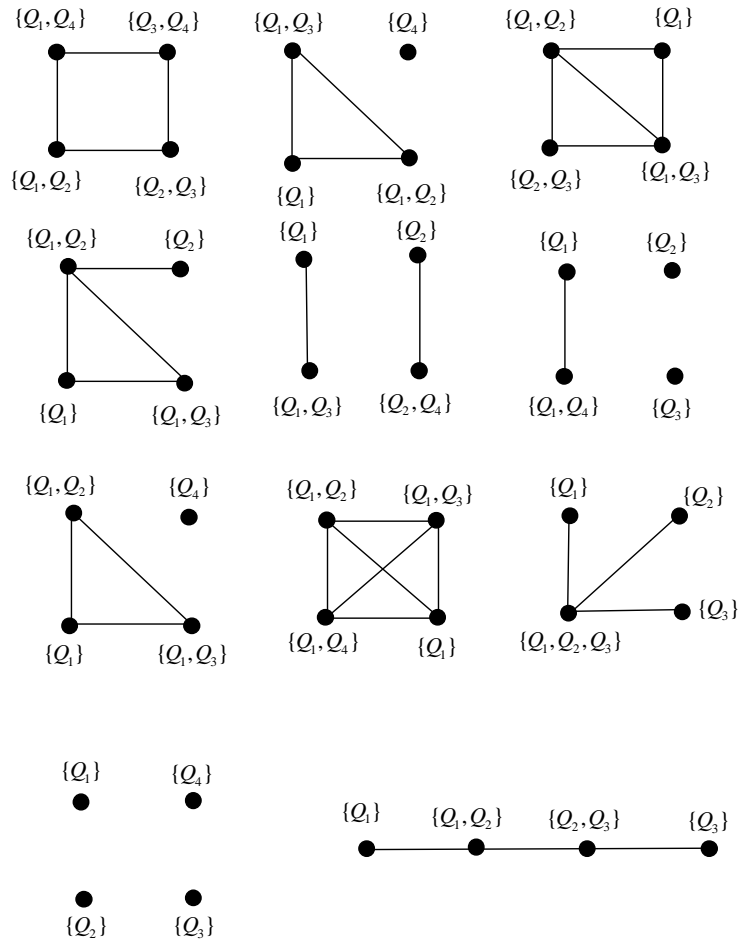}
\caption{The 11 Non-Isomorphic Graphs on 4 Vertices with
Corresponding Family Representations}\label{v4}
\end{figure}

We proceed by mathematical induction from $n=4$. First note that
given any positive integer $n$,
\begin{align*}
\lfloor n^2/4\rfloor = \lfloor (n-1)^2/4\rfloor + \lfloor
n/2\rfloor.
\end{align*}

Hence we have to show that from $G^{(n-1)}$ to $G^{(n)}$, at most
$\lfloor n/2\rfloor$ more cliques are needed. We have the following
cases:

{\bf{Case~1}}: In case $G^{(n)}$ has a vertex $v$ of degree $\leq
\lfloor n/2\rfloor$, then first we delete $v$ and all edges incident
with $v$. Then by induction hypothesis, we partition the resulting
graph with at most $\lfloor (n-1)^2/4\rfloor$ cliques $K_2$ or
$K_3$. Then from $G^{(n-1)}$ to $G^{(n)}$ we need only to use the
edges joining the deleted vertex $v$ to other vertices of $G^{(n)}$,
and then give rise to at most $\lfloor n/2\rfloor$ more cliques as
$K_2$. Clearly the resulting clique partition of $G^{(n)}$ still
satisfies (1).

{\bf{Case~2}}: On the contrary, every vertex of $G^{(n)}$ is of
degree $>\lfloor n/2\rfloor$. Let $x$ be the vertex with the minimum
degree $t$, and set $t=\lfloor n/2\rfloor +r$, where $r>0$. Let $x$
be adjacent to the vertices $y_1,...,y_t$ and $G^{(t)}$ be the
subgraph of $G^{(n)}$ induced by $\{y_1,...,y_t\}$.

We claim that $G^{(t)}$ has $r$ edges and no two of which have a
common vertex. Assume that $G^{(t)}$ has only $r-1$ such edges (note
that it is similar to show the case $G^{(t)}$ has less than $r-1$
such edges), say
\begin{align*}
\{y_1,y_2\},\{y_3,y_4\},...,\{y_{2r-3},y_{2r-2}\}.
\end{align*}
By $t=\lfloor n/2\rfloor +r=d(x)\leq n-1$, we know that $r\leq
\lfloor n/2\rfloor$ and thus $t\geq 2r$. Thus we may pick $y_{2r-1}$
from $\{y_1,...,y_t\}$.

By hypothesis, $y_{2r-1}$ has degree $\geq \lfloor n/2\rfloor +r$.
But it could be adjacent to at most $2r-2$ of the vertices
$y_1,...,y_{2r-2}$ and to at most $n-t$ of the vertices not in
$G^{(t)}$, hence the degree of $y_{2r-1}$ is at most
\begin{align*}
(2r-2)+(n-t)&=(2r-2)+(n-(\lfloor n/2\rfloor +r))\\
&=(n-\lfloor n/2\rfloor -2)+r\\
&<\lfloor n/2\rfloor+r.
\end{align*}

However note that $\lfloor n/2\rfloor +r$ is the minimum degree,
hence $y_{2r-1}$ is adjacent to some other vertex, say $y_{2r}$, in
$G^{(t)}$ and
\begin{align*}
\{y_1,y_2\},\{y_3,y_4\},...,\{y_{2r-3},y_{2r-2}\},\{y_{2r-1},y_{2r}\}
\end{align*}
are $r$ edges in $G^{(t)}$ and no two of which have a common vertex.

We remove these $r$ edges from $G^{(n)}-x$. Partition the resulting
graph with at most $\lfloor (n-1)^2/4\rfloor$ cliques and (1) is
satisfied. Then the $\lfloor (n-1)^2/4\rfloor$ cliques together with
the triangles
\begin{align*}
\{x,y_1,y_2\},\{x,y_3,y_4\},...,\{x,y_{2r-1},y_{2r}\}
\end{align*}
and the edges
\begin{align*}
\{x,y_k\}, \mbox{ where } 2r+1\leq k\leq t,
\end{align*}
form a clique partition, which uses at most
\begin{align*}
\lfloor (n-1)^2/4&\rfloor +r+(t-2r)\\
&=\lfloor (n-1)^2/4\rfloor -r+(\lfloor n/2\rfloor +r)\\
&=\lfloor n^2/4\rfloor
\end{align*}
cliques.

Note that according to our convention in this paper, we need to use
at least one trivial clique, even for each isolated vertex in the
clique partition of the graph $G^{(n)}-x$ with the $r$ edges
\begin{align*}
\{y_1,y_2\},\{y_3,y_4\},...,\{y_{2r-3},y_{2r-2}\},\{y_{2r-1},y_{2r}\}
\end{align*}
removed. Thus the resulting clique partition of $G^{(n)}$, obtained
from that of $G^{(n)}-x$ with the $r$ edges removed, must agree with
the requirement (1) of our theorem in the respect that for any two
vertices $u,v$ in $G^{(n)}$,
\begin{align*}
\{Q_i\mid u \in Q_i &\in \{Q_1,...,Q_N\}\}\\
\quad &\neq \quad \{Q_i\mid v\in Q_i\in \{Q_1,...,Q_N\}\}.
\end{align*}
%\end{proof}

Last we show that the number $\lfloor n^2/4\rfloor$ cannot be
replaced by any smaller number by giving the following example. Let
$n=2k$ or $2k+1$, we consider the complete bipartite graphs
$K_{k,k}$ and $K_{k,k+1}$, which have $2k$ and $2k+1$ vertices,
respectively. Clearly these two graphs have no triangle and their
numbers of edges are
\begin{align*}
k^2=\lfloor (2k)^2/4\rfloor=\lfloor n^2/4\rfloor, \mbox{ if } n=2k,
\end{align*}
and
\begin{align*}
k(k+1)=\lfloor (2k+1)^2/4\rfloor=\lfloor n^2/4\rfloor, \mbox{ if }
n=2k+1.
\end{align*}

Hence $K_{k,k}$ and $K_{k,k+1}$ always require $\lfloor
n^2/4\rfloor$ cliques for a clique partition.
\end{proof}

Now we introduce the one-to-one correspondence between multifamily
representations and clique partitions of a multigraph $M$ as
following.

Given a multigraph $M^{(n)}=(V(M),E(M),q)$, we first construct a
clique partition
\begin{align*}
\mathcal{Q}=\{Q_1,...,Q_p\}
\end{align*}
Then with each clique $Q_k$ we associate an element $e_k$ and with
each vertex $v_{\alpha}$ we associate a set
$S_{\mathcal{Q}}(v_\alpha)$ of elements $e_k$, where
\begin{align*}
e_k\in S_{\mathcal{Q}}(v_\alpha) \Leftrightarrow v_\alpha \in Q_k,
\end{align*}
i.e., $S_{\mathcal{Q}}(v_\alpha)$ is the collection of elements for
which the corresponding cliques contain $v_\alpha$. Thus we obtain
\begin{align*}
\mathcal{F}(\mathcal{Q})\equiv\{S_{\mathcal{Q}}(v):v\in V(M)\}.
\end{align*}
Then clearly
\begin{align*}
\textbf{S}(\mathcal{F}(\mathcal{Q}))\equiv \bigcup_{v\in
V(M)}S_\mathcal{Q}(v)
\end{align*}
contains $p$ elements. And
\begin{align*}
|S_\mathcal{Q}(v_\alpha)\cap
S_{\mathcal{Q}}(v_{\beta})|=q(v_{\alpha},v_{\beta}),
\end{align*}
since there is exactly $q(v_\alpha,v_\beta)$ cliques simultaneously
containing the two vertices $v_\alpha,v_\beta$. Thus we have
constructed a multifamily representation
$$\mathcal{F}(\mathcal{Q})=\{S_{\mathcal{Q}}(v):v\in V(M)\}$$ from the clique partition $\mathcal{Q}$ of $M$, where
$$|\textbf{S}(\mathcal{F}(\mathcal{Q}))|\equiv |\bigcup_{v\in V(M)}S_\mathcal{Q}(v)|=p=|\mathcal{Q}|.$$

Conversely, given a multifamily representation
$\mathcal{F}=\{S_1,...S_n\}$ of $M$ with vertex set
$V(M)=\{v_1,...,v_n\}$, where $S_\alpha$ correspond to the set
attaching to $v_\alpha$, then we can also construct a clique
partition of $M$ by the following way.

Let $$\textbf{S}(\mathcal{F})\equiv
\bigcup_{\alpha=1}^nS_\alpha=\{e_1,...,e_p\}.$$ For each fixed $e_k$
in $\textbf{S}(\mathcal{F})$ we form a clique $Q_\mathcal{F}(e_k)$
using those vertices $v_\alpha$ such that the set $S_\alpha$
attaching to it contains $e_k$. Clearly each $Q_\mathcal{F}(e_k)$ is
indeed a clique of $M$. Thus we obtain
$$\mathcal{Q}(\mathcal{F})=\{Q_\mathcal{F}(e_1),...,Q_\mathcal{F}(e_p)\}.$$
And
\begin{align*}
q(v_\alpha,v_\beta)&=|S_\alpha \cap S_\beta|\\
= \mbox{the number of cliques in } \mathcal{Q}(\mathcal{F}) &\mbox{
simultaneously containing } v_\alpha,v_\beta,
\end{align*}
since each element in $S_\alpha$ exactly represents a clique in
$\mathcal{Q}(\mathcal{F})$ containing $v_\alpha$. Thus we have
constructed a clique partition $\mathcal{Q}(\mathcal{F})$ of $M$
from the multifamily representation $\mathcal{F}$ of $M$, where
$$|\mathcal{Q}(\mathcal{F})|=p=|\bigcup_{\alpha=1}^nS_\alpha|\equiv |\textbf{S}(\mathcal{F})|.$$

Thus we have established a one-one correspondence between
multifamily representations and edge clique partitions of the
multigraph $M$.

From above we know that $\omega_m(M)=cp(M)$. In particular we may
consider the simple graphs as special classes of multigraphs:

\begin{thm}
Let $G$ be a graph. Then we have $\omega_m(G)=cp(G)$.
\end{thm}

If we are given a graph $G^{(n)}$, then by
Theorem~\ref{similarErdos} we may obtain a clique partition
$\mathcal{Q}$ with cardinality less or equal to $\lfloor
n^2/4\rfloor$, agreeing with the requirement (1). Then by the above
method we may obtain a representation
$\mathcal{F}(\mathcal{Q})=\{S_\mathcal{Q}(v):~v\in V(G)\}$ of $G$
consisting of distinct sets. Thus we have the following theorem.

\begin{thm}\label{uppbound}
Let $G$ be a graph. Then $\omega(G^{(n)})\leq \lfloor n^2/4\rfloor$.
\end{thm}

Again considering the two complete bipartite graphs $K_{k,k}$ and
$K_{k,k+1}$, one can easily see that the bound $\lfloor
n^2/4\rfloor$ in Theorem~\ref{uppbound} is sharp.

\section{\bf {Greedy Clique Decomposition of Graphs}}
One may not be satisfied with the above theorem and would rather ask
that how to obtain a representation of $G^{(n)}$ using at most
$\lfloor n^2/4\rfloor$ elements.
%In fact, McGuinness
%\cite{mcguinness} regarded this problem in the original theorem in
%Erd\"{o}s et al. \cite{erdos} as the motivation.
S. McGuinness \cite{mcguinness} proved the following theorem, which
solved a conjecture by P. Winkler~\cite{win}:

\begin{thm}
Every greedy clique decomposition of an $n$-vertex graph uses at
most $\lfloor n^2/4\rfloor$ cliques.
\end{thm}

In the theorem, the so-called \textit{clique decomposition} is a
clique partition of the edge set, and \textit{greedy clique
decomposition} of a graph $G^{(n)}$ means an ordered set
$\textbf{Q}=\{Q_1,...,Q_m\}$ such that each $Q_i$ is a maximal
clique in $G-\bigcup_{j<i}E(Q_j)$, where $G-\bigcup_{j<i}E(Q_j)$ is
the subgraph of $G$ obtained by deleting all edges in the edge
subset $\bigcup_{j<i}E(Q_j)$ while leaving all vertices in $G$
preserved.

For a representation $\mathcal{F}$ of $G$, we referred as
\textit{monopolized elements} to those elements in
$\textbf{S}(\mathcal{F})$ which appear in only one member of
$\mathcal{F}$. Here we prove the following main theorem, as a
variant of S. McGuinness's result:

\begin{thm}\label{similarMcGuinness}
Every representation $\mathcal{F}$ of $G^{(n)}$ with $n\geq 4$
obtained from $\mathcal{F}(\textbf{Q})$, where $\textbf{Q}$ is any
greedy clique decomposition of $G^{(n)}$ by successively attaching
monopolized elements to the sets which repetitiously occur in
$\mathcal{F}(\textbf{Q})$,
%where note that by this method, provided
%that there are $k$ sets in $\mathcal{F}(\textbf{Q})$ being identical
%with each other, we need only $k-1$ monopolized elements rather than
%$k$,
uses at most $\lfloor n^2/4\rfloor$ elements.
\end{thm}

Before proving the theorem, we need the following lemma:

\begin{lem}\label{mainlemma}
Let $\mathcal{Q}$ be an edge clique partition of a graph $G$, then
we have that if $\mathcal{F}(\mathcal{Q})=\{S_\mathcal{Q}(v):~v\in
V(G)\}$ has two identical sets, say $S_\mathcal{Q}(u)$ and
$S_\mathcal{Q}(v)$, then the clique $Q_{uv}$ in $\mathcal{Q}$
simultaneously containing $u,v$ is a maximal clique in $G$. Note
that $Q_{uv}$ has $u$ and $v$ as its monopolized elements, that is,
$u,v$ are in no clique of $\mathcal{Q}$ except $Q_{uv}$.
%implying that all vertices adjacent
%to $u,v$ in $G$ are just all vertices in $Q_{uv}-\{u,v\}$.
\end{lem}

\begin{proof}
If there is a clique $Q'$ properly containing $Q_{uv}$ in $G$, say
vertex $w$ being in $Q'$ but not in $Q_{uv}$, then no clique in
$\mathcal{Q}$ can simultaneously contain the three vertices $u,v,w$.
Thus the clique in $\mathcal{Q}$ simultaneously containing $u,w$
doesn't contain $v$ and the clique in $\mathcal{Q}$ simultaneously
containing $v,w$ doesn't contain $u$, and therefore we must have
$S_\mathcal{Q}(u)\neq S_\mathcal{Q}(v)$.

If $u$, say, belongs to one clique $Q''$ in $\mathcal{Q}$ other than
$Q_{uv}$, then there is a vertex, say $u'$, adjacent to $u$ and not
in $Q_{uv}$. In case that $u'$ is not adjacent to $v$ we must have
$S_\mathcal{Q}(u)\neq S_\mathcal{Q}(v)$. In case that $u'$ is
adjacent to $v$, then no clique in $\mathcal{Q}$ can simultaneously
contain $u,v,u'$. Thus the clique in $\mathcal{Q}$ simultaneously
containing $u,u'$ doesn't contain $v$ and the clique in
$\mathcal{Q}$ simultaneously containing $u',v$ doesn't contain $u$,
and therefore we must have $S_\mathcal{Q}(u)\neq S_\mathcal{Q}(v)$.
\end{proof}

Then we are in a position to proceed the proof of
Theorem~\ref{similarMcGuinness}:
\begin{proof}
We use induction on $n$.

When $n=4$, it is easy to draw all the eleven different graphs on
four vertices, and to check that every representation of each of
them uses at most $\lfloor n^2/4\rfloor$ elements.
%As for $n=5$, the
%number of non-isomorphic graphs is sufficiently large to make a
%reduction being desired.

For the case $n=5$, note that $\lfloor 5^2/4\rfloor -\lfloor
4^2/4\rfloor =6-4=2$ and therefore we have two new elements in
proceeding from $n=4$ to $n=5$. We have the following four cases:

{\bf{Case~1}}: If $G^{(5)}$ has one vertex with degree 2 or less,
then we reduce $G^{(5)}$ to $G^{(4)}$ by deleting this vertex and
all edges incident to it. Note that this vertex form a maximal
clique in $G^{(5)}$ along with some edge in $G^{(4)}$ only if
$G^{(5)}$ is one of 13 non-isomorphic graphs in Figure~\ref{hollow},
where hollow circle denote this vertex and dashed lines denote the
edges incident to it.
\begin{figure}
\centering
\includegraphics[width=0.7\textwidth]{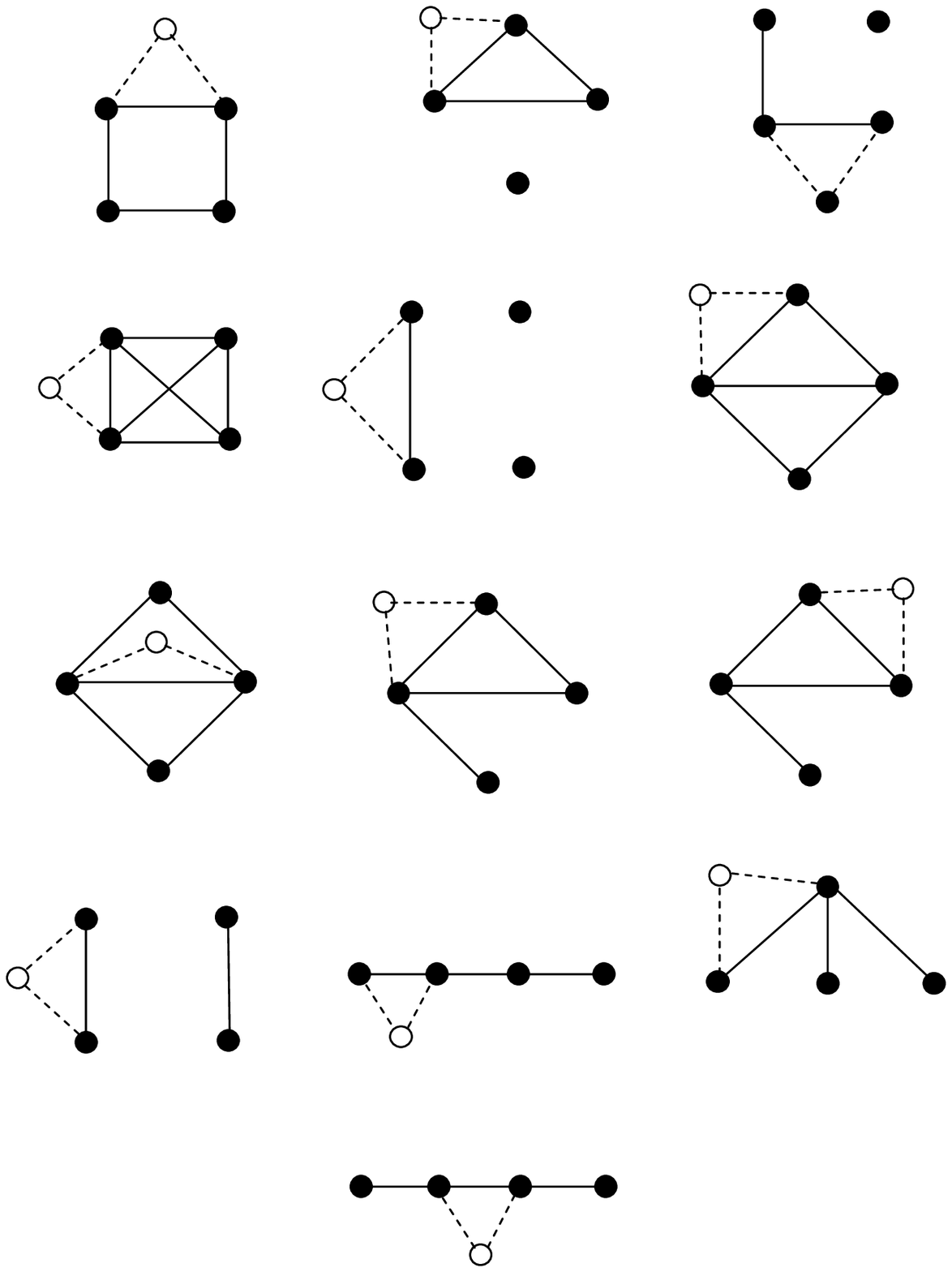}
\caption{One Case for Representations of Graphs on 5
Vertices}\label{hollow}
\end{figure}
It is easy to check that every representation of each one of them
uses at most $\lfloor 5^2/4\rfloor =6$ elements.

{\bf{Case~2}}: As for the case that there is no maximal clique in
$G^{(5)}$ simultaneously containing this vertex and some edge in
$G^{(4)}$, then in any greedy clique partition of $G^{(5)}$ we must
use all the edges incident to this vertex as members of this greedy
clique partition. Thus in this case, we may at first take a
representation of $G^{(4)}$, and then go back to $G^{(5)}$ using the
available two new elements to represent at most two edges incident
to this vertex. Then we may confirm that in this case all
representations of $G^{(5)}$ use at most $\lfloor 5^2/4 \rfloor =6$
elements.

{\bf{Case~3}}: As for the case that there is no edge in $G^{(5)}$
incident to this vertex, we may at first take a representation of
$G^{(4)}$ and then go back to $G^{(5)}$ using one new monopolized
element.

{\bf{Case~4}}: Due to above, now we need to consider only those
graphs on 5 vertices for which every vertex has degree greater than
or equal to 3. There are only three such graphs and they are easy to
be checked. (Please see Figure~\ref{v5}) Thus the case $n=5$ is
done, and we have proved the theorem for $n=4$ and $n=5$.\\

\begin{figure}[h]
\centering
\includegraphics[width=0.7\textwidth]{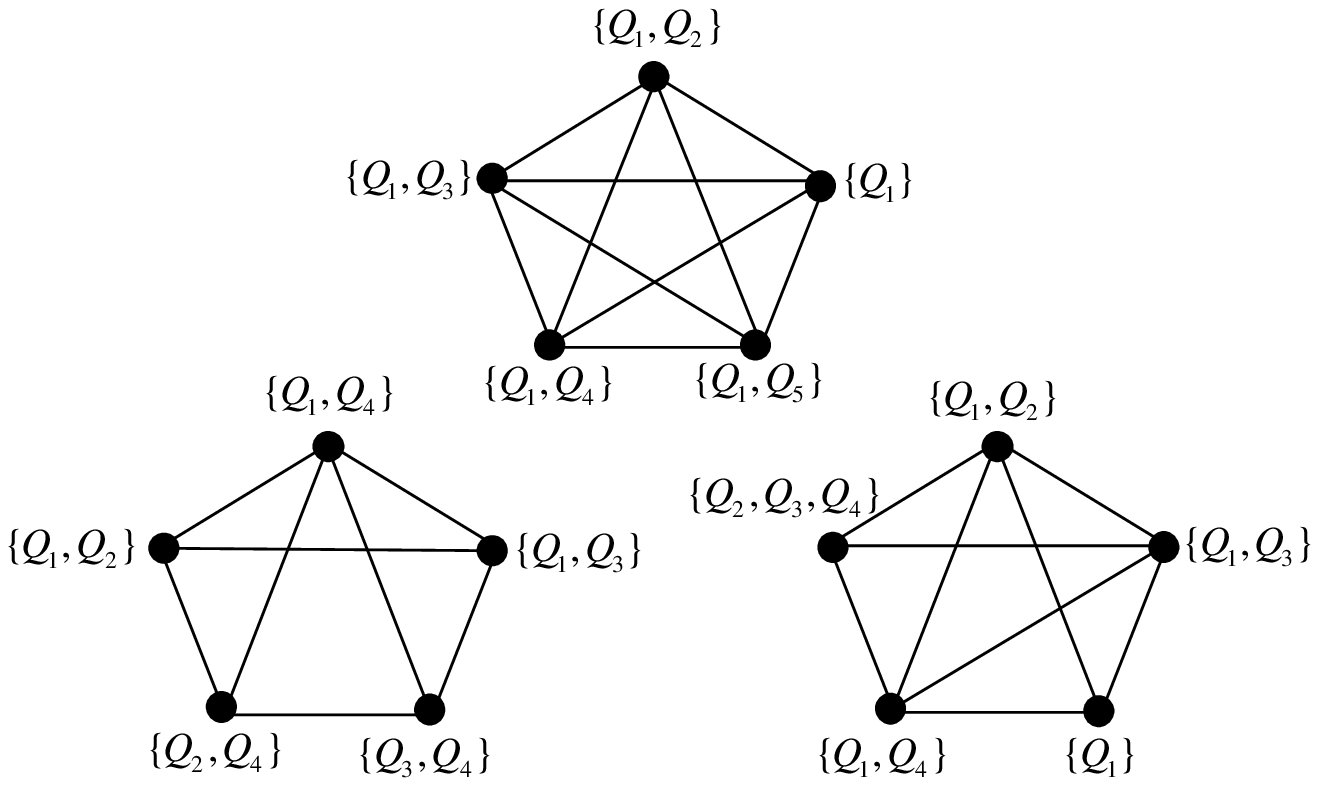}
\caption{Graphs on 5 Vertices Whose Vertices Have Degree $\geq 3$
with Corresponding Representations}\label{v5}
\end{figure}

Now let $\mathcal{F}$ be a representation of $G^{(n)}$ with $n\geq
6$ derived from $\mathcal{F}(\textbf{Q})$, where
$\textbf{Q}=\{Q_1,...,Q_m\}$ is a greedy clique partition of
$G^{(n)}$. Note that deleting $Q_j$ from the set $\textbf{Q}$ leaves
a greedy clique partition of $G-E(Q_j)$.

In case that each $Q_j$ has at least three edges, we have $m\leq {n
\choose 2}\big/3<n^2/6$. Assume for the time being that every $Q_i$
has exactly three edges, that is, is exactly a triangle. Now if
every triangle in $\textbf{Q}$ has at most one of its three vertices
of degree 2, then by Lemma~\ref{mainlemma} we do not need to use any
monopolized element for this greedy clique partition. If there is a
triangle in $\textbf{Q}$ with at least two of its three vertices of
degree 2, then recall that $G^{(n)}$ have at least six vertices, two
vertices of degree 2 in this triangle make $m$ to be less than or
equal to $\left({n \choose 2}/3\right)-2< (n^2/6)-2$. Thus although
we might need two more monopolized elements for this triangle, yet
in the same time we also have two less cliques (as $K_3$) in
$\textbf{Q}$. Besides, if there is a clique of cardinality $3+r$
where $r>0$ in $\textbf{Q}$, then despite that maybe we need $r$
more monopolized elements for this clique, yet in the same time by
the fact that ${3+r \choose 2}\geq 3(r+1)$ we also have $r$ less
cliques (as $K_3$)in $\textbf{Q}$. Note that ${3+r \choose 2}$ is
the number of edges in a clique of cardinality $3+r$ and $3(r+1)$ is
the total number of edges in $r+1$ triangles. In fact, we may need
rather $r+1$ or $r+2$ than $r$ more monopolized elements for this
clique of cardinality $3+r$. By Lemma~\ref{mainlemma}, we need to
use $r+2$ more monopolized elements for this clique only when either
this clique is an isolated clique or $G^{(n)}$ is itself a clique.
For the latter case, we use $n$ elements to represent $G^{(n)}$ and
note that $n^2/6\geq n$ for $n\geq 6$. As for the former case, we
lose all the edges joining this isolated clique to all the vertices
not on this isolated clique, therefore we lose at least 5 edges from
the calculated ${n \choose 2}$ edges and hence further lose at least
two cliques from the calculated $n^2/6$ cliques (as $K_3$). By
Lemma~\ref{mainlemma}, we need to use $r+1$ more monopolized
elements for this clique only when this clique has exactly $r+2$
vertices of degree $(3+r)-1$. In this case, this clique has a vertex
$v$ adjacent to one vertex, say $v'$, not in this clique, and all
vertices in this cliques other than $v$ are not adjacent to $v'$.
Therefore in $G^{(n)}$ we have $r+2\geq 3$ less edges than complete
graph $K_n$, and thus we have still one less triangle in
$\textbf{Q}$. Now we have brought to the conclusion that in case
that each $Q_j$ has at least three edges, we never use more than
$n^2/6$ elements to form a representation of $G^{(n)}$. Now we have
justified assuming some $Q_j$ is an edge $xy$. In case that
$d(x)=d(y)=1$, we may first take a representation of $G^{(n)}-x-y$
by the method of Theorem~\ref{similarErdos} using at most $\lfloor
(n-2)^2/4\rfloor$ elements, and then use two new elements for the
isolated edge $xy$ to form a representation for $G^{(n)}$ with at
most $\lfloor n^2/4\rfloor$ elements. Thus in this case every
representation of $G^{(n)}$ derived from $\mathcal{F}(\textbf{Q})$,
where $\textbf{Q}$ is any greedy clique partition of $G^{(n)}$, uses
at most $\lfloor n^2/4\rfloor$ elements.

As for the case that one of $x,y$ has degree more than one, in any
representation of $G^{(n)}$ we can not use any monopolized element
on $x$ or $y$. Now let $R$ consist of the members of
$\textbf{Q}-\{Q_j\}$ that are incident to $x$, and $S$ consist of
those incident to $y$. Then the set
$$\textbf{Q}'=\textbf{Q}-(R\cup S\cup \{Q_j\})$$ is a greedy clique partition of
$$G'\equiv (G^{(n)}-x-y)-\bigcup_{Q_i\in R \mbox{ or } S}E(Q_i),$$ except possibly leaving some isolated vertices
in $G'$ uncovered by any members of $\textbf{Q}'$. Recall that for
the present case, in $\mathcal{F}$ we never use any monopolized
element on $x,y$. Now if we can prove that every monopolized element
in $\textbf{S}(\mathcal{F})$ is always necessary for deriving a
representation of $G'$ from $\mathcal{F}(\textbf{Q}')$, then by
induction hypothesis we prove that
\begin{align}
|\mathcal{Q}(\mathcal{F})-(R\cup S\cup \{Q_j\})|\leq \lfloor
(n-2)^2/4\rfloor\tag{2}.
\end{align}

If in $\mathcal{F}$ we used one monopolized element on some vertex
$v$ not belonging to any member of $R\cup S$, then in
$\mathcal{F}(\textbf{Q})$, the set $S_\textbf{Q}(v)$ must be
identical with some $S_\textbf{Q}(u)$ where $u$ is also a vertex not
belonging to any member of $R\cup S$. Since both $u$ and $v$ do not
belong to any member of $R\cup S$, then
$S_{\textbf{Q}'}(u)=S_{\textbf{Q}'}(v)$ in
$\mathcal{F}(\mathcal{Q}')$. Thus this monopolized element is
necessary for deriving a representation of $G'$ from
$\mathcal{F}(\textbf{Q}')$.

If in $\mathcal{F}$ we used one monopolized element on some vertex
$v$ belonging to one member, say $Q_v$, of $R\cup S$. Then in
$\mathcal{F}(\textbf{Q})$, the set $S_\textbf{Q}(v)$ must be
identical with some $S_\textbf{Q}(u)$ where $u$ is also a vertex
belonging to $Q_v$. Now by Lemma~\ref{mainlemma} $v$ must have all
its neighbors in $Q_v$. Thus $v$ is an isolated vertex in $G'$. Thus
this monopolized element is necessary for deriving a representation
of $G'$ from $\mathcal{F}(\textbf{Q}')$. Thus we have proved the
statement (2).

Now it suffices to prove that $$|R\cup S|\leq n-2,$$ since
$$n-2\leq \lfloor n^2/4\rfloor -\lfloor (n-2)^2/4\rfloor -1.$$
We prove this by choosing distinct vertices in $V(G)-\{x,y\}$ from
the vertex sets of the members of $R\cup S$. Note that since each
edge is covered exactly once in a clique partition, each $v\notin
\{x,y\}$ appears once in $R$ if $v$ is adjacent to $x$ and once in
$S$ if $v$ is adjacent to $y$. Consider $Q_1\in R$. If $Q_1$
contains a vertex $v$ not adjacent to $y$, then we choose such a $v$
for $Q_1$. If all vertices in $Q_1$ are adjacent to $y$, then we
choose a vertex $v\in Q_1$ such that $vy$ belongs to the first
member of $\textbf{Q}$, say $Q_2$, which contains both $y$ and some
vertex of $Q_1$. Note that $Q_2$ is the only member of $S$
containing $v$.

Now we have two cases, that is, either that $Q_1$ precedes $xy$ in
$\textbf{Q}$ or that $xy$ precedes $Q_1$ in $\textbf{Q}$. For the
first case, since $Q_1$ and $xy$ are maximal while chosen, $Q_2$
must precedes $Q_1$ in $\textbf{Q}$ for otherwise from the
aforementioned hypothesis that all vertices in $Q_1$ are adjacent to
$y$ and $Q_1$ precedes $xy$ in $\textbf{Q}$, $Q_1$ should have
contained $y$ and hence $xy$. For the second case, since $xy$ is
maximal while chosen, one of $Q_1,Q_2$ precedes $xy$ or otherwise
$xy$ should have contained $v$. Thus in this case $Q_2$ precedes
$Q_1$ in $\textbf{Q}$. Note that in both cases, we have that $Q_2$
precedes both of $Q_1$, $xy$ in $\textbf{Q}$.

For the members of $S$, similarly as above choose vertices by
reversing the roles of $x$ and $y$.

In above we have shown that if $v$ belongs to some $Q_1\in R$ and to
some $Q_2\in S$, and $v$ is chosen for one of them, then the one for
which it is chosen occurs after the other one in the ordered set
$\textbf{Q}$. Hence no vertex is chosen twice. Thus we conclude that
$$|R\cup S|\leq n-2.$$
\end{proof}

\section{\bf{Conclusion Remarks}}

The edge clique partitions, as a special case of edge clique covers,
are served as great classifying and clustering tools in many
practical applications, therefore it is interesting to explore the
concept in more details.

One may work on the cases besides multifamily and family, say
antichain, uniform family etc. The greedy way to obtain these
variants also naturally gives the optimal upper bounds for the
corresponding intersection numbers.

\end{document}